\documentclass[12pt,a4paper]{amsart}
\usepackage{graphicx}
\usepackage{amsmath,amsfonts}
\usepackage{amsthm,amssymb,latexsym}
\usepackage[active]{srcltx}
\usepackage{natbib}

\newcommand{\XX}{a}

\newtheorem{theorem}{Theorem}[section]
\newtheorem{corollary}[theorem]{Corollary}
\newtheorem{lemma}[theorem]{Lemma}
\newtheorem{fact}[theorem]{Fact}
\newtheorem*{fact*}{Fact}
\newtheorem{proposition}[theorem]{Proposition}
\theoremstyle{definition}
\newtheorem{definition}[theorem]{Definition}

\newtheorem{example}[theorem]{Example}
\numberwithin{equation}{section}

\newcommand{\N}{\mathbb N}
\newcommand{\Z}{\mathbb Z}

\newcommand{\A}{\mathbb A}
\newcommand{\F}{\mathbb F}
\newcommand{\fp}{\F_{\!p}}
\newcommand{\fq}{\F_{\!q}}
\newcommand{\cfp}{\overline{\F}_{\!p}}
\newcommand{\cfq}{\overline{\F}_{\!q}}

\newcommand{\R}{\mathbb R}
\newcommand{\C}{\mathbb C}
\newcommand{\id}{{\text{id}}}
\newcommand{\exend}{\hfill $\lozenge$}

\newcommand{\norm}[1]{|| #1 ||}

\begin{document}

\title[Shifted varieties and discrete neighborhoods]{Shifted varieties and discrete neighborhoods around varieties}%
\author[von zur Gathen \&\ Matera]{
Joachim von zur Gathen${}^{1}$,
Guillermo Matera${}^{2,3,4}$}

\address{${}^{1}$B-IT\\
  Universit\"{a}t Bonn\\
  D-53113 Bonn}
  \email{gathen@bit.uni-bonn.de}

\address{${}^{2}$Instituto del Desarrollo Humano\\
Universidad Nacional de Gene\-ral Sarmiento, J.M. Guti\'errez 1150
(B1613GSX) Los Polvorines, Buenos Aires, Argentina}
\address{${}^{3}$National Council of Science and Technology (CONICET),
Ar\-gentina}
\address{${}^{4}$Departamento de Matem\'atica\\
Facultad de Ciencias Exactas y Naturales\\ Universidad de Buenos
Aires, Ciudad Universitaria, Pabell\'on I (1428) Buenos Aires,
Argentina}\email{gmatera@dm.uba.ar}

\thanks{GM is partially supported by the grants PIP
CONICET 11220130100598 and PIO CONICET-UNGS 14420140100027}%
%
%\subjclass{}%
%

\date{\today}%
%\dedicatory{}%
%\commby{}%
% ----------------------------------------------------------------
\begin{abstract}
In the area of symbolic-numerical computation within computer algebra,
an interesting question is how ``close'' a random input is to the ``critical'' ones,
like the singular matrices in linear algebra or the polynomials with multiple roots
for Newton's root-finding method.
Bounds, sometimes very precise, are known for the volumes over $\R$ or $\C$ of such neighborhoods
of the varieties of ``critical'' inputs; see the references below.

This paper deals with the discrete version of this question: over a finite field,
how many points lie in a certain type of neighborhood around a given variety?
A trivial upper bound on this number is (size of the variety) $\cdot$ (size of a neighborhood of a point).
It turns out that this bound is usually asymptotically tight, including for the singular matrices,
 polynomials with multiple roots, and pairs of non-coprime polynomials.

The interesting question then is: for which varieties does this
bound not hold? We show that these are precisely those that admit a
shift, that is, where one absolutely irreducible component is a
shift (translation by a fixed nonzero point) of another such
component. Furthermore, the shift-invariant absolutely irreducible
varieties are characterized as being cylinders over some base
variety.

Computationally, determining whether a given variety is shift-invariant turns out to be intractable,
namely NP-hard even in simple cases.

\end{abstract}

\maketitle
\section{Introduction}
In computer algebra, the area of symbolic-numerical computation has gained a lot of attention
in the past decades.
One question here is how ``close'' an input is to a set of ``critical'' instances.
Iterative numerical methods may not work on specific ``critical'' (or ``ill-posed'') inputs,
and may work badly on inputs ``close'' to these critical ones.
Sometimes this is described by \emph{condition numbers}, where large values
indicate closeness to criticality.
For some tasks of linear algebra, the singular (square) matrices are critical
and $1/|\text{determinant}|$ is the condition number.
In the Newton method for finding roots of (univariate) polynomials,
inputs with multiple roots are critical, and one might consider $1/|\text{discriminant}|$ as a condition number.
A well-studied issue, starting with \cite{schoenhage1985} and \cite{hribernigStetter1997},
is the \emph{approximate gcd} of a pair of univariate polynomials:
how close is it to a pair of non-coprime polynomials, with $1/|\text{resultant}|$ as a condition number.
Such condition numbers indicate the accuracy of solutions and the number of iterations until convergence.
% In this paper, we consider one variant of the various notions of closeness.

Much work about this question has focussed on individual inputs.
A more global question is: what is the probability for a uniformly random input (say, a random matrix)
to be $\epsilon$-close to some critical (that is, singular) one?

The general question of the measure (over $\R$ or $\C$) of such $\epsilon$-neighborhoods
(or $\epsilon$-tubes, tubular neighborhoods) of varieties is first considered in \cite{hot39} and \cite{wey39}.
In Smale's (\citeyear{sma81}) work on the efficiency of Newton's method, the size of $\epsilon$-neighborhoods
around the polynomials with multiple roots plays a major role.
\cite{ren87} extends this to solving general complex polynomial systems.
\cite{dem88} provides upper and lower bounds on the size of such neighborhoods of varieties,
and in particular for singular matrices and for polynomials with multiple roots, over $\R$ and $\C$.
\cite{belpar07} generalize and improve these findings.
The most precise result for matrices is the exact determination by \cite{edelman88,edelman92}
of the distribution function of the condition number.
Von zur Gathen \&\ Matera (\citeyear{gatmat18}) present upper and lower bounds on the size of $\epsilon$-neighborhoods
of the variety of decomposable univariate polynomials over $\R$ and $\C$.

Such results may be viewed as continuous analogs of the corresponding counting problem over finite fields,
replacing  the probability for a uniformly random input being ``in'' by being ``close to'' the variety in question.

The present paper turns this question around by determining the asymptotic size of
\emph{discrete neighborhoods} of varieties over a finite field.
This might be called \emph{discretizing a continuous version of a discrete problem}.
It turns out that this size is usually asymptotically close to a trivial upper bound.
Interestingly, the exceptions can be characterized as \emph{shifted varieties}.
Alas, the computational problem of determining whether a given variety belongs to this class
is intractable for growing input size.
%
%Our approach takes an affine variety $X \subseteq \cfp^n$ defined
%over a finite prime field $\fp$, where $\cfp$ is an algebraic
%closure of $\fp$, and some integer $h\geq 1$ and considers the
%discrete $h$-{neighbor\-hood} of the set of $\fp$-rational points
%$X(\fp)$ of $X$, consisting of those points in $\fp^n$ whose
%distance from
% $X(\fp)$ is not more than $h$, for a natural notion of ``distance''.
%Its size is at most $\# X(\fp) \cdot \# U_h$, where $U_h $ is the ball of radius $h$ centered at $0$.
%This paper addresses the question:
%\renewcommand{\theequation}{Q}
%\begin{equation}
%\label{question}
%\text{Is the neighborhood's size close to its upper bound?}
%\end{equation}

In more detail, we start with an odd prime $p$, the field $\fp =
\{-(p-1)/2, \ldots, (p-1)/2 \}$ with $p$ elements considered as a
subset of the integers $\Z$ with arithmetic modulo $p$, positive
integers $n$ and $h < p/2$, the $\infty$-norm $\norm{a} = \max_{1
\leq i \leq n} |a_i|$ for $a = (a_1, \ldots, a_n) \in \fp^n$ and the
ball $U_h = \{ a \in \fp^n \colon \norm{a} \leq h \}$ of radius $h$
and with $\# U_h = (2h+1)^n$ elements.

Given a system $f$ of polynomials in $\fp[x_1, \ldots, x_n]$,
we consider its affine variety $X = \{ f=0 \} \subseteq \cfp^n$,
 its rational points $X(\fp)$ that lie in $\fp^n$,
 and the ``discrete neighborhood'' $X(\fp)+U_h = \{ a+u \in
\fp^n \colon a \in X(\fp), u\in U_h\}$ around $X(\fp)$ of radius
$h$. Then
\setcounter{equation}{0}
\begin{equation}\label{eq:UpperBound}
\# (X(\fp)+U_h) \leq \#X(\fp) \cdot \# U_h.
\end{equation}
This paper addresses the question:
\begin{equation}
\renewcommand{\theequation}{Q}
\label{question} \text{Is the neighborhood's size close to this
upper bound?}
\end{equation}
%
%Question (\ref{question}) asks
Thus we ask for upper bounds on the (non-negative) difference
\setcounter{equation}{1}
\begin{equation}
\renewcommand{\theequation}{\thesection.\arabic{equation}}
\label{Delta}
\Delta =  \#X(\fp) \cdot \# U_h - \# (X(\fp)+U_h).
\end{equation}
A small such bound implies a lower bound on $\#(X(\fp)+U_h)$.

The central notion for understanding \eqref{question} turns out to
be the \emph{shift} of a variety, which is the translation by a
nonzero constant vector of the coordinates. If no absolutely
irreducible component with maximal dimension of $X$ is a shift of
another component, then the answer to (\ref{question}) is ``yes''.
For the opposite case, we exhibit examples where the answer is
``no''. {When $X$ is absolutely irreducible, the condition on shifts
turns out to be necessary and sufficient}
(Corollary \ref{character}).
Examples of \emph{shift-free} absolutely irreducible varieties include:
square matrices of rank bounded by some fixed number, non-squarefree polynomials,
pairs of non-coprime polynomials,
decomposable polynomials, and graphs of polynomials.

%Under mild assumptions, we show that
%it is close to the upper bound in some cases, and we also exhibit
%other cases where the gap is substantial.
% Testing whether the assumption holds for a given variety turns out to be computationally infeasible.

\section{Preliminaries}

Let $q$ be a power of a prime $p$ and $\fq$ a finite field with $q$
elements. We denote by $\cfq$ the algebraic closure of $\fq$ and by
$\mathbb{A}^n=\mathbb{A}^n(\cfq)$ the affine $n$-dimensional
space over $\cfq$. A nonempty subset $X\subset\mathbb{A}^n$ is an
affine subvariety of $\mathbb{A}^n$ (a variety for short) if it is
the set of common zeros in $\mathbb{A}^n$ of some set of polynomials in
$\cfq[x_1,\dots,x_n]$. Further, $X$ is an
$\fq$-variety (or $\fq$-definable) if it can be defined by polynomials in
$\fq[x_1,\dots,x_n]$. We will use the notations
 $\{f_1 = \cdots = f_s=0\}$ and $\mathcal{V}(f_1,\dots,f_s)$ to denote
the $\fq$-variety defined by $f_1,\dots,f_s$.

An $\fq$-variety $X\subset \A^n$ is $\fq$-irreducible if it cannot
be expressed as a finite union of proper $\fq$-subvarieties of $X$.
Further, $X$ is absolutely irreducible if it is $\cfq$--irreducible
as an $\cfq$--variety. Any $\fq$-variety $X$ can be expressed as a
non-redundant union $X=X_1\cup \cdots\cup X_m$ of irreducible
$\fq$-varieties, unique up to reordering,
which are called the {\em irreducible}
$\fq$-components of $X$. Some of them may be absolutely irreducible.

For an $\fq$-variety $X\subset\A^n$, its defining ideal $I(X)$ is
the set of polynomials in $\fq[x_1,\ldots, x_n]$ vanishing on $X$.
%The coordinate ring $\fq[X]$ of $X$ is the quotient ring
%$\fq[x_1,\ldots,x_n]/I(X)$. \todo{do we use this?}
The {\em dimension} $\dim X$ of an
$\fq$-variety $X$ is the length $r$ of a longest chain
$X_0\varsubsetneq X_1 \varsubsetneq\cdots \varsubsetneq X_r$ of
nonempty irreducible $\fq$-varieties contained in $X$.

The degree $\deg X$ of an irreducible $\fq$-variety $X$ is the
maximum number of points lying in the intersection of $X$ with a
linear space $L \subseteq \cfq^n$
 of codimension $\dim X$, for which $X\cap L$ is
finite. More generally, following \cite{Heintz83} (see also
\cite{Fulton84}), if $X=X_1\cup\cdots\cup X_m$ is the decomposition
of $X$ into irreducible $\fq$-components, then the degree of $X$ is
$$\deg X=\sum_{1 \leq i \leq m} \deg X_i.$$
The following {\em B\'ezout inequality} holds (see \cite{Heintz83},
\cite{Fulton84}, \cite{Vogel84}): if $X$ and $Y$ are
$\fq$-varieties, then
\begin{equation}\label{eq:BezoutInequality}
\deg (X\cap Y)\le \deg X \cdot \deg Y.
\end{equation}

In the following, we usually state explicit inequalities, but the spirit is
that the field size $q$ is (much) larger than the geometric quantities like $n$,
$\deg X$, and $h$, so that our bounds should be taken as asymptotics in $q$.
Thus an upper bound on $\Delta$ is  ``small'' if it is of smaller order in $q$ than the
arguments of $\Delta$.

We denote by $X(\fq)$ the set of $\fq$-rational points of an
$\fq$-variety $X\subseteq\A^n$, namely, $X(\fq)=X\cap \fq^n$. For $X$
of dimension $r$ and degree $d$, we let
$X= X_1\cup \cdots \cup X_{m}$ be the
decomposition of $X$ into irreducible $\fq$-components, and suppose
that $X_1,\ldots,X_\sigma$ are absolutely irreducible of dimension $r$
 and that
 $X_{\sigma+1},\ldots,X_m$ are not.
Then  $X_1,\ldots,X_\sigma$ provide the main contribution to $\# X(\fq)$
and we call them the \emph{essential components}.
 We write
\begin{align*}
d_i=\deg X_i \text{ for } 1\le i\le m \text{ and } D=\sum_{1 \leq i \leq \sigma}d_i.
\end{align*}
Then the following bounds on  $\#X(\fq)$ hold.
\begin{fact}
\label{generalUpperBounds}
\begin{enumerate}
\item
\label{eq:UpperBoundRationalPoints} $   \#X(\fq)\leq d q^r.$
\item
\label{th:EstimateRatPointsVariety}
If $q>d$ and $\sigma>0$, then
$$|\# X(\fq)-\sigma q^r|\le (D-1)(D-2)q^{r-{1}/{2}}+
(5D^{{13}/{3}}+d^2)q^{r-1}.
$$
\item
\label{lemma:RelativelyIrredVarieties}
If $X$ is an irreducible $\fq$-variety and not absolutely irreducible, then
$$\#X(\fq)\le d^2q^{r-1}/4.$$
\end{enumerate}
\end{fact}

Proofs of (\ref{th:EstimateRatPointsVariety}) and (\ref{lemma:RelativelyIrredVarieties})
are in  \cite{CaMa06}, Theorem 5.7 and Lemma 2.3.

\section{Neighborhoods around varieties}
Given a polynomial sequence
$f=(f_1,\ldots,f_s)$ in $\fp[x_1, \ldots, x_n]^s$ and the affine
$\fp$-variety $X = \{ f=0 \} \subseteq \A^n$, question (\ref{question}) is concerned with
the size of the ``standard neighborhood''
$$X(\fp)+U_h = \{ a+u \in \fp^n \colon a \in X(\fp), u\in U_h\}.$$
As $\#X(\fp) \cdot \# U_h$ is an optimal upper
bound, we concentrate on lower bounds, or, equivalently,
on upper bounds on the difference $\Delta$ from (\ref{Delta}).

\subsection{Generalized neighborhoods around varieties}

Most of this paper deals with the following more general problem: given an
$\fq$-variety $X= \{ f=0 \} \subseteq \A^n$, and a nonempty set
$U\subset\fq^n$, find lower bounds on
$$X(\fq)+U = \{ a+u \in \fq^n \colon a \in X(\fq), u\in U\}.$$

Since
\begin{equation}\label{eq:LowerBound}
\Delta = \#X(\fq) \cdot \# U - \#(X(\fq) + U) \leq \sum_{a \neq b \in
X(\fq)}\#((a + U) \cap (b + U)) ,
\end{equation}
it is sufficient to show upper bounds on the latter sum.
%
%For $a,b \in X(\fq)$ with $a \neq b$ and $v,w \in U$ with $a+v =
%b+w$, we have $b = a+v-w \in X(\fq)$. If $u = v-w\in U-U$, then $a+u
%\in X(\fq)$ with $0 \neq u \in U-U$.

We fix the following notation for an irreducible decomposition of
an $\fq$-variety $X = X_1 \cup \cdots \cup X_m \subset\A^n$ of dimension $r$:
\begin{equation}
\label{components}
\begin{aligned}
X_1,\ldots,X_m %\subset\A^n
\colon & \text{irreducible
$\fq$-components},  \\
 X_1,\dots,X_\sigma \colon &
\text{absolutely irreducible of dimension $r$},\\
X_{\sigma+1},\ldots,X_\rho \colon & \text{absolutely irreducible of dimension
less than $r$},\\
X_{\rho+1},\ldots,X_m \colon & \text{not absolutely
irreducible},
\end{aligned}
\end{equation}
with $0 \leq \sigma \leq \rho \leq m$. Recall that $
X_1,\dots,X_\sigma$ are the \emph{essential components}. According
to Fact \ref{generalUpperBounds}, the cardinality of the set
$X_j(\fq)$ for an essential component $X_j$
%$X_1, \ldots, X_\sigma$
is of order $q^r$, while that of the other components is at most of order $q^{r-1}$.
The following notions of shifts are central to our considerations.

\begin{definition}
\label{defShiftFree}
\begin{enumerate}
\item
For $0 \neq u = (u_1, \ldots, u_n) \in \A^n$ and $g \in \fq[x_1,$ $ \ldots,$ $x_n]$,
$g^{(u)} = g(x_1-u_1, \ldots, x_n-u_n)$ is $g$ {\em shifted by}
$u$.
\item
For $f \in \fq[x_1, \ldots, x_n]^s$ and $X = \{ f=0
\} \subseteq \A^n$, $f^{(u)} $ consists of the polynomials in $f$
shifted by $u$, and
 $X^{(u)} = \{ f^{(u)} = 0\} = \{a+u \colon a \in X\}$ is $X$ {\em shifted by} $u$.
\item
\label{u-shiftfree}
$X$ with decomposition (\ref{components}) is {\em essentially $u$-shift-free}
if none of its components $X_i$ equals $X_j^{(u)}$
for any $i, j$ with $1 \leq i,j \leq \sigma$. (The case $i=j$ is included, and
if $\sigma=0$, then $X$ is essentially $u$-shift-free.)
\item
\label{defShiftFree1}
For $U \subseteq \fq^n$, $X$ is {\em essentially $U$-shift-free} if it is essentially $u$-shift-free for all $u \in U$.
\end{enumerate}
\end{definition}

Thus $X^{(u)}$ is isomorphic to $X$, and if $X=Y^{(u)}$, then $X^{(-u)} = (Y^{(u)})^{(-u)} = Y$.
When $X$ is absolutely irreducible, so that $\sigma = m =1$ in (\ref{components}),
we leave out the word ``essentially''.

\begin{lemma}\label{lemma:NumberPointsInTwoDisks}
We have
$$
\sum_{a\not=b\in X(\fq)}\#((a+U)\cap (b+U))\le \#U \cdot \sum_{0\not=u\in
U-U}\#(X(\fq)\cap X^{(u)}(\fq)).
$$
\end{lemma}
\begin{proof}
For $a,b \in X(\fq)$ with $a \neq b$, $v,w \in U$ with $a+v =
b+w$,  and $u=v-w$, we have $u\neq 0$ and $b = a+u \in X(\fq)\cap X^{(u)}(\fq)$. Since any
$u\in U-U$ can be expressed in at most $\#U$ ways as $u=v-w$ with
$v,w\in U$, the lemma follows.
\end{proof}

The following result first establishes an upper bound on the difference $\Delta$ from
\eqref{eq:LowerBound} in terms of intersections of shifted
varieties.
Then we state, under a shift-freeness assumption, a bound which is small compared to the two
arguments of $\Delta$.

\begin{theorem}\label{th:EstimateGenTubeVarieties}
Let $X\subset\A^n$ be an $\fq$-variety of dimension $r$, degree $d<q$,
and decomposition (\ref{components}).
\begin{enumerate}
\item
\label{generalBound}
\begin{align*}
\Delta \le \# U & \cdot \bigl( \sum_{\mbox{${0\not=u\in U-U} \atop{1\leq i,j\leq \sigma} $}}
%\Delta \le \# U & \cdot \bigl( \sum_{{0\not=u\in U-U}} \sum_{1\leq i,j\leq \sigma}
 \#(X_i(\fq)\cap X_j^{(u)}(\fq))
 \\ &  + 2 (\#(U-U)-1) \sum_{\sigma < i \leq m} \# X_i(\fq) \bigr ).
\end{align*}
\item
\label{shiftFreeBound}
If furthermore $X$ is essentially $(U-U)$-shift-free, then

\begin{align*}
% |\#(X(\fq)+U)- \#X(\fq) \cdot \# U|
\Delta \le \#U (\#(U-U)-1) \cdot d^2q^{r-1}.
\end{align*}
\end{enumerate}
\end{theorem}
\begin{proof}
Inequality \eqref{eq:LowerBound} and Lemma
\ref{lemma:NumberPointsInTwoDisks} imply that
$$
\Delta \le \#U\sum_{0\not=u\in U-U}
\#(X(\fq)\cap X^{(u)}(\fq)).$$
Given $0\not=u\in U-U$, we estimate
$$\#(X(\fq)\cap X^{(u)}(\fq))=
\#\big(\bigcup_{1\leq i,j\leq m}X_i(\fq)\cap X_j^{(u)}(\fq)\big).$$
For $\sigma < i\le m$ and $1 \leq j\le m$, we have
$X_i(\fq)\cap X_j^{(u)}(\fq) \subseteq X_i(\fq)$. %, whose size is small by Fact ?? \ref{generalUpperBounds}.
Similarly, $X_i(\fq)\cap X_j^{(u)}(\fq) \subseteq X_j^{(u)}(\fq)$
holds for $\sigma < j\le m$ and all $i$. Thus all these
intersections are contained in $\bigcup_{\sigma < i\le m} (X_i(\fq)
\cup X_i^{(u)}(\fq))$. Together with $\# X_i(\fq) =
\#X_i^{(u)}(\fq)$ this yields the bound claimed in
(\ref{generalBound}).

For (\ref{shiftFreeBound}), we first have $2 \# X_i(\fq) \leq d_i d q^{r-1}$ for $\sigma < i \leq m$
by Fact \ref{generalUpperBounds}.
It remains to consider
$1\le i,j\le \sigma$. By hypothesis $X_i$ and $X_j^{(u)}$
are distinct absolutely irreducible varieties and their intersection
has dimension at most $r-1$ and degree at most $d_id_j$ by the
B\'ezout inequality \eqref{eq:BezoutInequality}. By Fact \ref{generalUpperBounds}
\eqref{eq:UpperBoundRationalPoints} we have for any $0 \neq u \in U-U$ that
\begin{align*}
\#\big(\bigcup_{1\leq i,j\leq \sigma} X_i(\fq)\cap
X_j^{(u)}(\fq)\big)&\le\sum_{1\leq i,j\leq \sigma} \#(X_i(\fq)\cap
X_j^{(u)}(\fq))\\&\le \sum_{1\leq i,j\leq \sigma} d_id_jq^{r-1}.
\end{align*}
\end{proof}

For $\sigma=0$ we have the following less precise result,
which follows from  Fact \ref{generalUpperBounds} \eqref{eq:UpperBoundRationalPoints} and
(\ref{lemma:RelativelyIrredVarieties}), and (\ref{eq:UpperBound}) (for general $U$).
\begin{corollary}
With hypotheses and notations as in Theorem
\ref{th:EstimateGenTubeVarieties}, assume further that $\sigma=0$.
Then
\begin{align*}
\#(X(\fq)+U)\le\#U \cdot d^2q^{r-1}/2.
\end{align*}
\end{corollary}

When $X$ is not shift-free, then $\Delta$ may be large,
%estimate of Corollary \ref{coro:EstimateTubeVarieties} may not hold,
as in the following example.
\begin{example}
\label{parallelHyperplanes} Let $p > d+2h$, $h\ge 1$, $n\geq 2$, $f
= x_1 \cdot (x_1-1) \cdots (x_1 - (d-1)) \in \fp[x_1, \ldots, x_n]$,
so that $d=\deg f$ and $X= \{ f=0\}$ is the union of $d$ parallel
hyperplanes $H_i = \{x_1 = i\}$ for $0 \leq i < d$ with distance 1
between ``neighbors'', and $X(\fp)+U_h = \bigcup_{-h \leq i < d+h}
H_i(\fp)$. $X$ is invariant under many shifts. Namely, $X=X^{(u)}$
for any $u \in \{0\} \times \fp^{n-1}$, and for $0 \leq i < j \leq
d$, $X_j = X_i^{(u)}$ for any $u \in \{j-i\} \times \fp^{n-1}$. We
have $\# X(\fp)= d p^{n-1}$, $\# (X(\fp)+U_h) = (d+2h) p^{n-1}$, and
$$
\Delta = p^{n-1} ( d(2h+1)^n - (d+2h)).
%\frac {\# (X(\fp)+U_h) } {\#X(\fp) \cdot \# U_h} = \frac {1+2h/d}
%{(2h+1)^n}.
$$
Since $\Delta \neq 0$, there is no ``small'' upper bound on $\Delta$,
namely of order less than ${n-1} = \dim X$ in $p$.
In particular, if $X$ is a single hyperplane, then $\# (X(\fp)+U_h)
= (2h+1)p^{n-1}$. Its difference $\Delta
=((2h+1)^n - (2h+1))p^{n-1}$
with $\#X(\fp) \cdot \# U_h$ is of the same order of magnitude in $p$
as its two arguments, that is, not ``small''.
\exend
\end{example}

\begin{example}
\label{independentVariables}
Generalizing Example \ref{parallelHyperplanes}, we consider a
variety $Y \subset \A^n$ whose defining polynomials
are independent of the variables $x_{n-m+1},$ $\ldots, x_n$, for
some $m$ with  $1 \leq m<n$. We may also consider them as elements of
$\fq[x_1,\ldots,x_{n-m}]$,
% and call them $f'_1, \ldots, f'_n$ under this perspective.
they define a variety $Y' \subseteq \A^{n-m}$, and $Y(\fq) = Y'(\fq)
\times \fq^{m}$.
We consider the embedding
\begin{equation}
\begin{aligned}
e & \colon \A^{n-m} \hookrightarrow \A^n  \\ \XX' & \mapsto (\XX',0,\ldots,0),
\end{aligned}
\end{equation}
take some $V \subseteq \fq^n$,
and let $V' = e^{-1} (V)$.
Then $Y+V = (Y'+V') \times \A^{m}$.

%Proof: $x+u \in Y+U \Rightarrow x+u = (x',x''+u'') + (u',0) \in (Y'+U') \times A^{n-t}$.
%$((x'+u'),v) \in (Y'+U') \times A^{n-t} \Rightarrow ((x'+u'),v) = (x',v) + (u',0) \in Y+U$.
%
If we write $\Delta' = \#Y'(\fq) \#V' - \#(Y'(\fq) + V')$
and assume that $V = V' \times V''$ for some $V'' \subseteq \fq^{m}$ with
$\# V'' \geq 2$, then
\begin{equation}
\label{fewVariables} \Delta = q^{m} \bigl( \#(Y'(\fq)+V') \cdot
(\#V'' -1) + \Delta' \#V'' \bigr).
\end{equation}
We will modify this reasoning in Theorem \ref{notShiftFree} to show that under a certain condition,
no ``small'' upper bound on $\Delta$ exists.
\exend
\end{example}

\begin{example}
\label{boundedRank}
For $m,n\in\N$ and $s<\min\{m,n\}$, we consider the
``determinantal'' $\fq$-variety $M_s$ of matrices in $\A^{m\times
n}$ of rank at most $s$. It is well-known that $M_s$ is absolutely
irreducible with
$$r=\dim M_s=s(m+n-s),\quad d=\deg M_s=
\prod_{0\leq i < n-s}\frac{\binom{m+i}{s}}{\binom{s+i}{s}};$$
see, e.g., \cite[Proposition 1.1]{BrVe88} for the first assertion
and \cite[Example 19.10]{Harris92} for the second one.
A simple calculation reveals that those factors decrease monotonically with growing $i$,
so that the term for $i=0$ dominates and $d \leq \binom{m}{s}^{n-s}$.

In view of Theorem \ref{th:EstimateGenTubeVarieties}, we check that
$M_s$ is not shift-invariant. Let $u\in\fq^{m\times
n}\setminus\{0\}$ and consider $M_s^{(u)}$. As the zero matrix $0$
belongs to $M_s$, if $0+u=u$ is in $M_s$, then $t=\mathrm{rank}\,
u\le s$. Let $C\in\fq^{m\times m}$ be an invertible matrix such that
the last $m-t$ rows of $C\cdot u$ are equal to zero. Let
$A\in\fq^{m\times n}$ be a matrix whose first $t$ rows and last
$m-s-1$ rows are zero, and the remaining $m-t-(m-s-1)=s-t+1$ rows,
together with the first $t$ rows of $C \cdot u$, are linearly
independent. Such an $A$ exists, since $t+s-t+1 = s+1 \leq
\min\{m,n\}$. Then
\begin{align*}
\mathrm{rank}\,A&=s-t+1\le s,\\
\mathrm{rank}(C\cdot u+A)&=t+s-t+1=s+1.
\end{align*}
It follows that $C^{-1}A\in M_s$ and $C^{-1}(C\cdot
u+A)=u+C^{-1}A\notin M_s$, so that $M_s\not=M_s^{(u)}$, and thus
$M_s$ is $\fq^{m \times n}$-shift-free. Applying Theorem
\ref{th:EstimateGenTubeVarieties} we obtain for $U \subseteq
\fq^{m\times n}$
\begin{equation}
\label{boundRank}
\Delta \leq \# U \# (U-U) d^2 q^{r-1}.
%\hfill\lozenge
\end{equation}
\hfill$\lozenge$
\end{example}

As a further example, we consider the variety of decomposable univariate polynomials.
For a univariate polynomial $f=a_dx^d+\cdots+a_1x+a_0$ in a
polynomial ring $R[x]$ over a ring $R$ %$\fq[u_1,\ldots,u_t][x]$
and $k\in\N$, its {\em $k$th Hasse derivative} $\mathcal{D}^{(k)}f$
is
\begin{equation}
\label{defHasse}
\mathcal{D}^{(k)}f=\sum_{k \leq i \leq d} \binom{i}{k}a_ix^{i-k}.
\end{equation}
Since $\binom i k \binom {i-k} \ell = \binom{k+\ell} \ell \binom i {k+\ell}$
in the usual ranges for binomial coefficients,
we have
$\mathcal{D}^{(k)}\circ \mathcal{D}^{(\ell)} =
\binom {k+\ell} \ell \mathcal{D}^{(k+\ell)}$.

\begin{example}
\label{ex:decomposables}
A univariate polynomial $f = f_n x^n + \cdots + f_0 \in F[x]$ of degree $n$ over a field $F$ is
\emph{decomposable} if there exist $g,h \in F[x]$ of degrees $\ell, m \geq 2$, respectively, with $f = g \circ h$.
Then $n = \ell m$, and denoting their coefficients by $g_i$ and $h_j$, respectively, we also have
for the monic (leading coefficient 1) and original (constant coefficient 0, graph containing the origin)
polynomial $h' = h_m^{-1} (h-h_0)$:
$$
f = g \circ h = g \circ ((h_m x+ h_0) \circ h' ) = (g \circ (h_m x +
h_0)) \circ h'.
$$
We may thus assume that $h$ is monic original. Then $g_\ell = f_n$.
We might further normalize $f$ into the monic original $f' = f_n^{-1}(f-f_0)$, so that in
$f' = g' \circ h'$ all three polynomials are monic original, with the appropriate $g'$,
but do not use this here.

All such polynomials $f$, $g$, and $h$ are parametrized by their
coefficients in $\A^{n+1}$, $\A^{\ell+1}$, and $\A^{m-1}$,
respectively. The Zariski closure of the image of the
\emph{composition map} $\gamma\colon \A^{\ell +m} \rightarrow \A^{n+1}$
with $(g,h) \mapsto g \circ h$ is the set of decomposable
polynomials. This is an absolutely irreducible closed affine
subvariety $C_{n,\ell}$ of $\A^{n+1}$, of dimension $\ell +m$ and
with degree $d \leq \ell^{\ell+m-2}$; see \cite{gatmat18}.

In the remainder of this example, we assume that
$F$ is a finite
field $\fq$ of characteristic greater than $n$ and that $q$ is sufficiently large (compared to $n$).

We want to show that $C_{n,\ell}$ is $\fq^n$-shift-free. It is
sufficient to exhibit for every nonzero $u \in \A^{n+1}$ some $f \in
C_{n,\ell}$ so that $u+f \not \in C_{n,\ell}$. Addition here is the
standard coefficient-wise addition of polynomials.

By
the chain rule, the Hasse derivative $\mathcal{D} f = \mathcal{D}^{(1)} f$ of any $f=g\circ h$ with $\deg
f=n$, $\deg g=\ell$ and $\deg h=m$ has a factor $\mathcal{D}h$ of
degree $m-1$. We consider the set $\fq[x]_{< n}$ of polynomials of
$\fq[x]$ of degree at most $n-1$. We claim that the set of
$f\in\fq[x]_{< n}$ having a factor in $\fq[x]$ of degree $m-1$
is Zariski closed. Indeed, fix a factorization pattern
$1^{\lambda_1}\cdots(m-1)^{\lambda_{m-1}}$ for polynomials of degree
$m-1$, where $\lambda_1,\ldots,\lambda_{m-1}\in\Z_{\ge 0}$ are such
that $\lambda_1+\cdots+(m-1)\lambda_{m-1}=m-1$. For each $i\in\N$,
let $s_i\in\fq[x]$ be the product of all irreducible polynomials
of $\fq[x]$ of degree $i$. Then $f$ has a factor in $\fq[x]$ with
factorization pattern $1^{\lambda_1}\cdots(m-1)^{\lambda_{m-1}}$ if
and only if $\gcd(f,s_i^{n-1})$ has degree at least $\lambda_i$ for
$1\le i <m$. Further, the latter is equivalent to the vanishing
of the first $\lambda_i$ subresultants of $f$ and $s_i^{n-1}$ for
$1\le i < m$. This shows that the set of elements of $\fq[x]_{< n}$
having a factor in $\fq[x]$ with factorization pattern
$1^{\lambda_1}\cdots(m-1)^{\lambda_{m-1}}$ is Zariski closed.
Considering all possible factorization patterns for polynomials of
degree $m-1$, the claim follows.

Since the set of all $f\in\fq[x]_{\le n}$ as above
 is Zariski dense in $C_{n,\ell}$ and each
$\mathcal{D}f$ has a factor in $\fq[x]$ of
degree $m-1$, we conclude that the derivatives $\mathcal{D}f$ of all
$f\in C_{n,\ell}$ satisfy this closed condition.

We consider a nonzero $u \in \fq^{n+1}$, supposing first that $\deg u=n$
and $\mathcal{D}u(0)\not=0$. Any
$f=\sum_{0 \leq i \leq \ell}\lambda_ix^{im}$ with
$\lambda_0,\ldots,\lambda_\ell \in\fq$ belongs to $C_{n,\ell}$ and it suffices
to prove that there exists such an $f$ with $u+f\notin
C_{n,\ell}$. When $\lambda_0,\ldots,\lambda_\ell$ vary over $\fq$, the set
of polynomials
$$\mathcal{F}:=\bigl\{\mathcal{D}u+\sum_{1 \leq i \leq \ell}
im\lambda_ix^{im-1}:\lambda_0,\ldots,\lambda_\ell\in\fq\bigr\}$$
constitutes a linear family with prescribed coefficients in the
sense of, e.g., \cite[\S 3.5]{MuPa13}. Arguing by contradiction,
assume that $u+f\in C_{n,\ell}$ for any $f$ as above. Denote by
$\mathcal{N}=\{0,\ldots,n-1\}\setminus \{m-1,2m-1,\ldots,\ell m-1\}$
the set of exponents corresponding to the monomials having a
prescribed value. Then $\gcd(\mathcal{N})=1$, and
%\cite{Stepanov87} shows that for sufficiently large $q$,
%Stepanov: at least $1/n q^m + O(q^{m/2}, where $m=\# \mathcal{N}$.
%See MullenPanario,
\cite[Theorem 1]{Cohen72} shows that for sufficiently large $q$,
there exists a
polynomial $f$ so that $\mathcal{D}(u+f)\in\mathcal{F}$ is irreducible in
$\fq[x]$. According to our previous remarks, this contradicts our
assumption $u+f\in C_{n,\ell}$.

For the remaining case, where $\deg u<n$,
we add a fixed term $\lambda_nx^n$ with $\lambda_n\not=0$ and let the
remaining terms vary, while if $\mathcal{D}u(0)=0$, we make a
similar argument considering the Taylor expansion of $\mathcal{D}u$
in powers of $x-\alpha$ and the set of elements
$f=\sum_{0 \leq i \leq \ell} \lambda_i(x-\alpha)^{im}$ with
$\lambda_0,\ldots,\lambda_\ell\in\fq$, for a suitable
$\alpha\in\fq\setminus\{0\}$.

Thus $C_{n,\ell}$ is $\fq^n$-shift-free and the estimate of Theorem
\ref{th:EstimateGenTubeVarieties} (\ref{shiftFreeBound}) applies.
\exend
\end{example}

\section{Shift-invariant varieties}\label{subsec:ShiftInvariantVarieties}

In order to apply Theorem \ref{th:EstimateGenTubeVarieties} (\ref{shiftFreeBound})
to an
%an absolutely irreducible
$\fp$-variety $X \subseteq \A^n$,
the critical point is to check whether $X$ is $(U-U)$-shift-free.
%for any $0 <\norm{w}\le 2h$.
We say for some $u\in\fq^n\setminus\{0\}$ that $X$ is {\em
$u$-shift-invariant} if $X = X^{(u)}$. $X$ is \emph{shift-invariant}
if it is $u$-shift-invariant for some nonzero $u \in \fq^n$. $X$ is
a \emph{cylinder in the direction of $u$} if for any $\XX\in X$ and $t\in
\cfq$, $\XX+tu \in X$. We have the following characterization of
shift-invariance.
\begin{proposition}\label{prop:CharactShiftInvariantVarieties}
Let $p>d$, $q$ a power of $p$, let $X\subset\A^n$ be an $\fq$-variety of degree
$d$ and $u\in\fq^n\setminus\{0\}$. Then $X$ is $u$-shift-invariant if
and only if $X$ is a cylinder in the direction of $u$.
\end{proposition}
\begin{proof}
Suppose that $X$ is invariant under a shift
$u\in\fq^n\setminus\{0\}$. Let $\XX$ be an arbitrary point of $X$ and
consider the line $\ell_\XX=\{\XX+tu:t\in\cfq\}$. Since $X$ is invariant
under the shift $u$, it is also invariant under $2u, 3u, \ldots, (p-1)u$.
Thus
$$\#(X\cap\ell_\XX)\ge \#\{\XX+tu:t\in\fp\}= p>d.$$
If $\dim(X\cap\ell_\XX)=0$, then by the B\'ezout inequality
\eqref{eq:BezoutInequality} we would have
$$\#(X\cap\ell_\XX)=\deg(X\cap\ell_\XX)\le \deg X=d,$$
which contradicts the previous inequality. It follows that
$$0<\dim(X\cap\ell_\XX)\le\dim \ell_\XX=1,$$
so that $\dim(X\cap\ell_\XX)=1$. The fact that the variety
$X\cap\ell_\XX$ of dimension 1 is contained in the absolutely
irreducible variety $\ell_\XX$ of dimension 1 shows that
$X\cap\ell_\XX=\ell_\XX$, that is, $\ell_\XX\subseteq X$. Since this holds
for any $\XX\in X$, we conclude that $X$ is a cylinder in the
direction of $u$.

The converse assertion is clear.
\end{proof}

We can reformulate the condition of shift-invariance as follows.
\begin{corollary}\label{coro:CharactShiftInvariantVarietiesWithIdeals}
With hypotheses as in Proposition
\ref{prop:CharactShiftInvariantVarieties}, $X$ is shift-invariant if
and only if there exists
%linearly independent linear forms
%$y_1,\ldots,y_n\in\cfp[x_1,\ldots,x_n]$ such that if
an invertible map
$L \colon \A^n\to\A^n$
of linear forms in $\fq[x_1, \ldots, x_n]$ such that
%is the mapping $\Phi(x)=(y_1,\ldots,y_n)$, then
$L(X)=Y\times\A^1$ for some $\fq$-variety $Y\subseteq \A^{n-1}$. If
this is the case and $I(Y) \subset\cfq[y_1,\ldots,y_{n-1}]$ is the
ideal of $Y$, then $I(Y\times\A^1)=I(Y) \, \cfq[y_1,\ldots,y_n]$.
\end{corollary}
\begin{proof}
We assume that $X=X^{(w)}$ with some nonzero $w \in \fq^n$. For ease
of presentation, we apply a coordinate permutation $C$ so
that $(Cw)_n \neq 0$, and now assume $w_n\neq 0$. We define the
vector of linear forms
\begin{equation}
\label{linForm}
L =
\big(x_1- \frac{w_1 x_n}{w_n},  \ldots , x_{n-1}- \frac{w_{n-1} x_{n}}{w_n}, \frac{x_n}{w_n} \big) \in (\fq[x_1, \ldots, x_n])^n
\end{equation}
and also denote the induced mapping as
$$
L \colon \A^n \rightarrow \A^n.
$$
The linear forms in $L$ are linearly independent. % and $L(w) = (0,\ldots,0,1)$.
Let $N$ be $L$ followed by the projection to the first $n-1$ coordinates,
and $Y = N(X)$.
Thus $N(w) = 0$.
We claim that $L(X) = Y \times \A^1$.

The inclusion ``$\subseteq$'' is clear. So let $b \in Y$ and $c\in \A^1$, and $b=N(a)$
for some $a\in X$.
As in the proof of Proposition \ref{prop:CharactShiftInvariantVarieties}, the line
$\{a+t w \colon t\in\A^1\}$ is contained in $X$. Then
\begin{align*}
L\big(a+(c- \frac{a_n}{w_n})w\big)&
=\big(N(a+(c-\frac{a_n}{w_n})w),(a_n+(c-\frac{a_n}{w_n})w_n)/w_n \big)\\
&=\big(N(a),\frac{a_n}{w_n}+c-\frac{a_n}{w_n})\big)=(b,c),
\end{align*}
which shows the claim. The invertible map in the proposition is $L \circ C$.

  Finally, the identity
$I(Y\times\A^1)=I(Y) \, \cfq[y_1,\ldots,y_{n-1}]$ is a standard fact on
ideals of varieties.
\end{proof}

In the next result, a subset $U\subset\fq^n$ is called {\em closed
under shifts to zero} if for any $u\in U$, replacing any coordinate
of $u$ by zero yields an element of $U$. We remark that the standard
neighborhood $U_h\subset\fp^n$ is closed under shifts to zero.
Finally, we recall $\Delta$ from (\ref{eq:LowerBound}).

\begin{theorem}
\label{notShiftFree}
Let $p$ be a prime, $p > d \geq 1$, $U \subseteq \fp^n$ closed under
shifts to zero, and $X$ be an absolutely irreducible $\fp$-variety
of dimension $r$ and degree $d$ which is not $(U-U)$-shift-free.
Furthermore, let
\begin{equation}
\label{alpha} \alpha = d^2+ (5d^{{13}/{3}}+d^2
\#(U-U))p^{-{1}/{2}}
\end{equation}
and assume that $p \geq 4\alpha^2$.
Then
\begin{equation}
\Delta \geq p^r/2.
\end{equation}
\end{theorem}
\begin{proof}
The mapping given by $X \mapsto X^{(u)}$ for an $\fp$-variety $X$ constitutes an
action of the additive group $\fp^n$ on the set of $\fp$-varieties,
since $(X^{(u)})^{(u')} = X^{(u+u')}$. We let $B \subseteq U-U$ be a
basis of the subgroup generated by the $u\in U-U$ with $X =
X^{(u)}$. This subgroup is an $\fp$-vector space of some dimension $m$ and we
write $B = \{ b_1, \ldots, b_m \}$.
We take the invertible linear map $L_1 \colon \A^n \rightarrow
\A^{n-1} \times \A^1$ from (\ref{linForm})
with $w=b_1$,
ignoring for ease of presentation the possibly required coordinate permutation.
Thus $L_1(X) = Y_1 \times \A^1$ and $L_1(b_1)
= (0, \ldots ,0,1)$, for some subvariety $Y_1$ of $\A^{n-1}$. Also,
$N_1$ is $L_1$ followed by the projection to the first $n-1$
coordinates.

We claim that $V_1 = N_1(U)$ is closed under shifts to zero.
So let $u = (u_1, u_2, \ldots, u_n) \in U$, $v = (v_1, v_2, \ldots, v_{n-1}) = N_1(u)$ be an arbitrary element of $V_1$,
and consider annihilating its first coordinate.
Since $U$ is closed under shifts to zero, also $u' = (0, u_2, \ldots, u_{n-1}, 0) \in U$ and
$(0, v_2, \ldots, v_{n-1}) = N_1(u') \in V_1$.

If $m \geq 2$, we claim that $Y_1^{(N_1(b_2))} = Y_1$. Writing $L_1(b_2) = (b_2', b_2'')$ with
$b_2' = N_1(b_2) \in \A^{n-1}$ and $b_2'' \in \A^1$, we have
\begin{align*}
Y_1 \times \A^1 & = L_1(X) = L_1(X^{(b_2)}) = (L_1(X))^{(L_1(b_2))} \\
& = (Y_1 \times \A^1)^{(b_2', b_2'')} = Y_1^{(b_2')} \times \A^1.
\end{align*}
Thus $Y_1 = Y_1^{(b_2')}$, and we can again apply Corollary
\ref{coro:CharactShiftInvariantVarietiesWithIdeals} to find $L_2
\colon \A^{n-1} \rightarrow \A^{n-1}$ and $Y_2 \subseteq \A^{n-2}$
with $L_2(Y_1) = Y_2 \times \A^1$ and $L_2(b_2') = (0, \ldots,
0,1)$. This yields an invertible linear map $M_2 = L_2 \times \id
\colon \A^n \rightarrow \A^n$ with $(M_2\circ L_1)(X) = Y_2 \times
\A^2$. The last two entries of
%$(M_2\circ L_1)(b_1,b_2)$
$((M_2\circ L_1)(b_1),(M_2\circ L_1)(b_2))$ have the lower
antitriangular form
\begin{equation}
\label{triangular}
\left(
\begin{tabular}{lr}
0 & 1\\
1 & *
\end{tabular}
\right),
\end{equation}
with  some value $*$ below the antidiagonal.

Continuing in this way for a total of $m$ steps, we obtain an
invertible linear map $M = M_m \circ M_{m-1} \circ \cdots \circ L_1
\colon \A^n \rightarrow \A^n$, where $M_j = L_j \times \id_{j-1}$
and $\id_k \colon \A^k \rightarrow \A^k$ is the identity map, with
$M(X) = Y \times \A^m$ for some variety $Y$ in $\A^{n-m}$.
Furthermore, $Y$ is absolutely irreducible of dimension $r-m$ and
degree $d$. The last $m$ columns of
%$M \, (b_1, \ldots, b_m)$
$(M(b_1), \ldots, M(b_m))$ have a shape as in (\ref{triangular}),
namely lower antitriangular with 1 on the antidiagonal, zeroes above
it, and arbitrary values below it. We let $V \subseteq \fp^{n-m}$ be
the projection of $ M(U)$ to the first $n-m$ coordinates. Applying
inductively the argument for $V_1$ from above, it follows that also
$V$ is closed under shifts to zero.

We claim that $Y$ is $(V-V)$-shift-free.
So let $v\in V-V$ with $Y = Y^{(v)}$ and $z=M^{-1} (v,0)$. Then
$$
M(X) = Y \times \A^m = Y^{(v)} \times \A^m = (Y\times \A^m)^{(v,0)}
= M(X)^{(M(z))}= M(X^{(z)}),
$$
and hence $X = X^{(z)}$.
For the last equation in the above, we have for any $a\in \A^n$:
\begin{align*}
a \in M(X)^{(M(z))}  &\Longleftrightarrow a-M(z) \in M(X) \Longleftrightarrow M^{-1}(a) -z \in X\\
& \Longleftrightarrow M^{-1}(a) \in X^{(z)} \Longleftrightarrow a \in  M(X^{(z)}).
\end{align*}
Furthermore, we claim that $z \in U-U$ and
$z_{n-m+1} = \cdots = z_n =0$, and use downward induction on $m$ to
show this. We first consider the situation $m=1$ in Corollary
\ref{coro:CharactShiftInvariantVarietiesWithIdeals}, so that
$V=N(U)$, and write $v = v_1 - v_2$ with $v_1, v_2 \in V \subseteq
\fp^{n-1}$.
Since $V$ is closed under shifts to zero, also $(v_i,0) \in V$
and $z_i = L^{-1} (v_i,0) \in U$ for $i \in \{1, 2\}$.
$L$ is a linear map and therefore $z = z_1 - z_2 \in U-U$.
Thus the claim follows in this situation, and
in general by induction. %\todo{spell out induction?}

Thus $z$ is a linear combination, say $z = \sum_{1\leq i \leq m} \lambda_i b_i$
with all $\lambda_i \in \fp$,
of the basis vectors $b_i$, and also $M(z) = \sum_{1\leq i \leq m} \lambda_i M(b_i)$.
The last $m$ coordinates of $M(z)$ are zero by the above claim and those of the $M(b_i)$ have antitriangular
shape, with 1 on the diagonal, so that all $\lambda_i$ are zero.
It follows that $z=0$ and $v=0$, and therefore $Y$ is $(V-V)$-shift-free.

We now follow the reasoning in Example \ref{independentVariables}.
Thus we write
$\Delta' = \#Y(\fp) \#V - \#(Y(\fp) + V)$ and  $k = \# U / \# V$.
Then $\# X(\fp) = p^m \# Y(\fp)$
and
\begin{equation}
\label{X+U}
M(X+U) = M(X) + M(U) = (Y \times \A^m) + M(U) = (Y+V) \times \A^m.
\end{equation}
For the last equation, we first take some $b\in Y$, $c \in \A^m$, and $u = (u_1, \ldots, u_n) \in M(U)$.
Then $v = (u_1, \ldots, u_{n-m}) \in V$ and
$$
(b,c) + u = ( (b+v, c+(u_{n-m+1}, \ldots, u_n) ) \in (Y+V) \times \A^m.
$$
For the reverse inclusion, we take $b \in Y$, $v\in V$, and any
$u = (u_1, \ldots, u_n) \in M(U)$ with $(u_1, \dots, u_{n-m}) = (v_1, \dots, v_{n-m})$.
Then
$$
(b+v, c) = (b, c-(u_{n-m+1}, \ldots, u_n)) + (v,u) \in (Y \times \A^m) + M(U).
$$
 In particular, we have $\# (X(\fp) + U) =  p^m \#(Y(\fp) + V)$.

Since $b_1 \in U-U$, we also have $-b_1 \in U-U$
and $N_1(-b_1) = N_1(b_1) = 0$.
Thus $\#V < \#U$ and $\#(V-V) < \# (U-U)$.

Since $Y$ is $(V-V)$-shift-free and absolutely irreducible of dimension $r-m$ and degree $d$,
Corollary \ref{coro:EstimateGenTubeVarieties} below implies that
$$
\#(Y(\fp)+V) > p^{r-m} {\#V} (1 - \alpha p^{-{1}/{2}}).
$$

Now $\Delta' \geq 0$ and we have
\begin{equation}
\label{fewVariables2}
\begin{aligned}
\Delta & = \# X(\fp) \#U - \#(X(\fp) + U) \\ & = p^m \bigl( k \Delta' + (k -1) \#(Y(\fp)+V)  \bigr) \\
& \geq p^m (k -1) \#(Y(\fp)+V) \\
& > (\# U - \# V) (p^r - \alpha p^{r-{1}/{2}}) \geq p^r - \alpha
p^{r-{1}/{2}} \geq p^r/2.
\end{aligned}
\end{equation}
\end{proof}

The condition $p \geq 4\alpha^2$ has a wide range of solutions, for example $d\geq 13$,
$\# U \leq d^\gamma$ for some real $\gamma \geq 1.5$, and $p\geq 9 d^{2+2 \gamma}$.
Then $\# (U-U) \leq d^{2\gamma}$, $5d^{{13}/{3}} \leq d^5 \leq
d^{2+2 \gamma}$, $d^2 \leq  5 6 d^{1+\gamma}$,
\begin{align*}
4 \alpha^2 & \leq 4 (d^2 + 2d^{2+2 \gamma} p^{-{1}/{2}})^2 \leq 4 (d^2 +2d^{2+2\gamma} d^{-1-\gamma}/3)^2 \\
& = 4 (d^2 + 2d^{1+\gamma}/3)^2 \leq 4 (3d^{1+\gamma}/2)^2 = 9
d^{2+2\gamma} \leq  p.
\end{align*}
The bound in
Theorem \ref{th:EstimateGenTubeVarieties} (\ref{shiftFreeBound})
is roughly $(\# U)^3 d^2 p^{r-1}$, if $\# (U-U)$ is about $(\# U)^2$.
For this to be less than the first argument $\# X(\fp) \# U \approx p^r \#U$ of $\Delta$,
we certainly need $p > (\# U)^2$, which explains the requirement $p > d^{2 \gamma}$ above.

We now come to a characterization of the absolutely
irreducible varieties for which the answer to Question (\ref{question}) is ``yes''.
The bounds on $\Delta$ in this paper depend on $q$ and $r$,
and also on the parameters
$n$, $U$, and various degrees. For such a function $g$ and $s\geq 0$, we write
$g \in O(q^s)$ and $g \in \Omega(q^s)$, respectively, if $g \leq c q^s$ and $g \geq c q^s$
hold with functions $c$ of the parameters, but independent of $q$ and $s$,
and which are positive for a large range of the parameters.

\begin{corollary}
\label{character}
% Let $q$ be a power of a prime $p > d \geq 1$.
Let $p$ be a prime, $p > d \geq 1$, and $U \subseteq \fp^n$ closed
under shifts to zero. Furthermore, we assume $p \geq 4 \alpha^2$ for
$\alpha$ from (\ref{alpha}).
%\todo{Also for q = power of p?}
Then for absolutely irreducible $\fp$-varieties $X$ of dimension $r$
and degree $d$, we have:
\begin{align*}
X \text{ is } (U-U)\text{-shift-free} & \Longrightarrow \Delta \in O(p^{r-1}), \\
% \Delta_2 \in O(p^{r-1/2}), \\
X \text{ is not } (U-U)\text{-shift-free} & \Longrightarrow \Delta \in \Omega(p^{r}).
%, \Delta_2 .
\end{align*}
\end{corollary}

\begin{proof}
The claim in the first line is in Theorem \ref{th:EstimateGenTubeVarieties} (\ref{shiftFreeBound}).
%and Corollary \ref{coro:EstimateGenTubeVarieties}.
The second line follows from
Theorem \ref{notShiftFree}.
\end{proof}

\section{Weil bounds, standard neighborhoods,\\and hypersurfaces}
\label{standardNeighborhood}

While both summands of $\Delta$ are defined in (\ref{eq:LowerBound})
in terms of $\# X$, the Weil bounds in Fact \ref{generalUpperBounds}
allow more specific numerical bounds depending just on the dimension
and degree of $X$; due to their generality, these are somewhat less
precise. The paper's introduction explains our original motivation
of dealing with standard neighborhoods $U_h = \{ a \in \fp^n \colon
\norm{a} \leq h \}$ over $\fp$. We spell out the consequences of our
more general results for this special case. Furthermore, we discuss
the particular case of hypersurfaces in more detail.

\subsection{Weil bounds}

Are the upper bounds on $\Delta$ in Theorem \ref{th:EstimateGenTubeVarieties}
``small'' in relation to the two arguments of $\Delta$?
This is not always the case, as shown in Corollary \ref{character}.
Furthermore, if $\sigma=0$ in \eqref{components}, the asymptotic behavior of $\#
X(\fq)$ does not have a simple description suitable for our purposes.
For a partial positive answer, we now rule out this case and substitute a numerical
approximation for $\# X(\fq)$.
Then the upper bound in Theorem \ref{th:EstimateGenTubeVarieties} (\ref{shiftFreeBound})
indeed turns out to be small.
\begin{corollary}\label{coro:EstimateGenTubeVarieties}
With hypotheses and notations as in Theorem
\ref{th:EstimateGenTubeVarieties} (\ref{shiftFreeBound}), assume further that $\sigma>0$
and denote $D=\sum_{1\leq i \leq\sigma} \deg X_i$.
Then
\begin{align*}
|\#(X(\fq)+&U)-\# U\cdot \sigma q^r|\le\#U\big(D^2q^{r-{1}/{2}}+
(5D^{{13}/{3}}+\#(U-U)d^2)q^{r-1}\big).
\end{align*}
\end{corollary}
\begin{proof}
By Fact \ref{generalUpperBounds} (\ref{th:EstimateRatPointsVariety}) we have
\begin{align*}
|\#X(\fq)\cdot\# U-&\# U\cdot\sigma q^r|\le
\#U\big(D^2q^{r-{1}/{2}}+ (5D^{{13}/{3}}+d^2)q^{r-1}\big).
\end{align*}
Theorem \ref{th:EstimateGenTubeVarieties} (\ref{shiftFreeBound}) and the triangle inequality
imply the claim.
\end{proof}

\subsection{Standard neighborhoods}

Throughout this subsection, we have a prime $p$, an $\fp$-variety $X \subseteq \A^n$ of dimension $r$, degree $d<p$,
and with decomposition (\ref{components}),
$D=\sum_{1 \leq i \leq\sigma} \deg X_i$,
an integer $h$ with $p > 2h \geq 2$, and the standard neighborhood $U_h =
\{ u \in \fp^n \colon \norm{u} \leq h \}$. Then
$$\#U_h=(2h+1)^n,\quad \#(U_h-U_h)= \# U_{2h} =(4h+1)^n.$$
Now $\Delta =  \#X(\fp) \cdot \# U_h - \# (X(\fp)+U_h)$ as in (\ref{Delta}).
If $X$ is essentially $U_{2h}$-shift-free, then the following bounds
are consequences of Theorem \ref{th:EstimateGenTubeVarieties}
(\ref{shiftFreeBound}) and Corollary
\ref{coro:EstimateGenTubeVarieties}.
\begin{corollary}
\begin{enumerate}
\item
%\begin{equation*}
$
\Delta  \leq ((2h+1)(4h+1))^nd^2p^{r-1},
$
\item
\label{standardWeil}
\begin{align*}
|\#(&X(\fp)+U_h)-(2h+1)^n\sigma p^r|\\\le
&(2h+1)^n\big(D^2p^{r-{1}/{2}}+
(5D^{{13}/{3}}+(4h+1)^nd^2)p^{r-1}\big).
\end{align*}
%\end{corollary}
\item
If furthermore $X$ is absolutely irreducible, then $D=d$ in (\ref{standardWeil}).
\end{enumerate}
\end{corollary}

Next we specialize to some of our examples.
For the determinantal variety $M_s$
of $m \times n$ matrices with rank at most $s$ from Example \ref{boundedRank},
% $q=p$ prime, $p > 2h\geq 2$, and the standard
%neighborhood $U_{h} \subseteq \fp^n$,
%by Corollary
%\ref{coro:EstimateVarietiesAbsIrred}
 we have
\begin{align*}
& \Delta \leq ((2h+1)(4h+1))^{mn}d^2p^{r-1},\\
& |\#(M_s(\fp)+U_h)-(2h+1)^{mn} p^r|\\
& \quad \le (2h+1)^{mn} d^2 \big(p^{r-{1}/{2}}+
(5d^{{7}/{3}}+(4h+1)^{mn})p^{r-1}\big).
%\hfill \lozenge
\end{align*}

The variety $C_{n,\ell}$ of decomposable polynomials from Example \ref{ex:decomposables} has degree
 $d \leq \ell^{\ell+m-2}$ and satisfies:
\begin{align*}
& \Delta \leq ((2h+1)(4h+1))^{n+1}d^2p^{\ell+m-1},\\
& |\#(C_{n,\ell}(\fp)+U_h)-(2h+1)^{n+1} p^r|\\
& \quad \leq (2h+1)^{n+1} d^2 \big(p^{r-{1}/{2}}+
(5d^{{7}/{3}}+(4h+1)^{n+1})p^{r-1}\big).
%\hfill \lozenge
\end{align*}

\subsection{Hypersurfaces}
For a hypersurface $X= \{ f=0 \}\subset\A^n$, where
$f\in\fq[x_1,\ldots,x_n]$ is squarefree, the decomposition
(\ref{components}) of $X$ corresponds to the factorization of $f$
into irreducible polynomials of $\fq[x_1,\ldots,x_n]$, the first
$\sigma = \rho$ many being absolutely irreducible. We assume that $d
= \deg X = \deg f <q$ and that $X$ is essentially
$(U-U)$-shift-free. With the usual notation in this section, we have
the following bounds.

\begin{corollary}
 \begin{enumerate}
 \item
 $\Delta \leq \#U \#(U-U) d^2q^{n-2}$,
 \item
\begin{align*}
|\#&(X(\fq)+U)-  \# U  \cdot \sigma q^{n-1}| \\
& \quad \le
\#U\big(D^2q^{n-{3}/{2}}+
(5D^{{13}/{3}}+\#(U-U)d^2)q^{n-2}\big).
%(5D^{\frac{13}{3}}+\#(U-U)d^2)q^{n-2}\big).
\end{align*}
\end{enumerate}
 For $q=p$ prime and $U=U_h$, we have
\begin{enumerate}
 \setcounter{enumi}{2}
\item
 $\Delta \leq ((2h+1)(4h+1))^nd^2p^{n-2}$,
 \item
\begin{align*}
|\#(&X(\fp)+U_h)-(2h+1)^n\sigma p^{n-1}|\\ & \quad \le
(2h+1)^n\big(D^2p^{n-{3}/{2}}+
(5D^{{13}/{3}}+(4h+1)^nd^2)p^{n-2}\big).
\end{align*}
\end{enumerate}
\end{corollary}

\section{Shift-invariant polynomials}

This section derives some properties of shift-invariant polynomials
and studies algorithmic aspects.

We use the Taylor expansion of multivariate
polynomials, employing the Hasse derivatives from (\ref{defHasse}). For
$f\in\fq[x_1,\ldots,x_n]$, its {\em $k$th partial (Hasse) derivative
$\mathcal{D}_{x_i}^{(k)}f$ with respect to $x_i$} is
$\mathcal{D}^{(k)}f$ for $f$ considered as an element of $R[x_i]$
with $R = \fq[x_1,\ldots,x_{i-1},x_{i+1},\ldots,x_n]$. Derivatives
with respect to different variables commute:
$\mathcal{D}_{x_i}^{(k)} \circ \mathcal{D}_{x_j}^{(\ell)} =
\mathcal{D}_{x_j}^{(\ell)} \circ \mathcal{D}_{x_i}^{(k)}$ for $i
\neq j$.
 For a multi-index
$s=(s_1,\ldots,s_n)\in\Z^n$, we write
$\mathcal{D}^{(s)}f=\mathcal{D}_{x_1}^{(s_1)}\circ \cdots\circ
\mathcal{D}_{x_n}^{(s_n)}$ if $s_1,\ldots,s_n \geq 0$, and
$\mathcal{D}^{(s)}f= 0$ otherwise.
For $s,t\in\N^n$, we have
\begin{equation}\label{eq:CompositionDiffOperators}
\mathcal{D}^{(s)}\circ
\mathcal{D}^{(t)}=\binom{s+t}{s}\mathcal{D}^{(s+t)}, \text{ where }
\binom{s+t}{s} = \prod_{1 \leq i \leq n} \binom{s_i+t_i}{s_i}.
\end{equation}

Furthermore, we let $\N_{\leq d}^n \subset\N^n$ be the set of indices
$s=(s_1,\ldots,s_n)$ with $|s|=s_1+\cdots+s_n\le d$.
%, and similarly $\N_{= d}^n \subset\N^n$.
For $s\in \N^n$ and $w = (w_1, \ldots, w_n) \in W^n$ for some ring $W$, we set
\begin{equation}
\label{powerProd}
w^s = \prod_{1 \leq i \leq n} w_i^{s_i}.
\end{equation}
With this terminology, we have the following version of the Taylor
formula: if $f\in\fq[x_1,\ldots,x_n]$ and
$y = (y_1,\ldots,y_n)$ are new indeterminates, then
\begin{equation}\label{eq:TaylorFormula}
f=\sum_{s\in\N^n_{\leq d}} \big( (\mathcal{D}^{(s)}f)(y) \big) \cdot (x-y)^s,
\end{equation}

The gradient of $f \in \fq[x_1,\ldots,x_n]$ is $\nabla
f = (\mathcal{D}^{(1)}_{x_1}(f), \ldots, \mathcal{D}^{(1)}_{x_n}(f))
\in \fq[x_1,\ldots,x_n]^n$.
We have the following consequence of shift invariance in terms of
derivatives.
\begin{corollary}\label{coro:CharactShiftInvariantVarietiesWithDerivatives}
With hypotheses as in Proposition
\ref{prop:CharactShiftInvariantVarieties}, if $X$ is $u$-shift-invariant for some
$u\in\fq^n\setminus\{0\}$, then for any $f\in I(X)$ and $\XX\in X$ we have
$(\nabla f)(\XX)\cdot u=0$.
\end{corollary}
\begin{proof}
Let $\XX\in X$ and $f\in I(X)$. Proposition
\ref{prop:CharactShiftInvariantVarieties} shows that the line
$\{\XX+tu:u\in\cfq\}$ is contained in $X$. As a consequence,
$f(\XX+tu)=0$ for any $t\in\A^1$. Then the corollary follows by the
chain rule.
\end{proof}

%Corollary \ref{coro:EstimateHypersurfacesAbsIrred} motivates the
In order to study the shift-invariance of $f\in\fp[x_1,\ldots,x_n]$
of degree $d$, let $y_1,\ldots,y_n$ be new indeterminates.
 By the Taylor formula \eqref{eq:TaylorFormula},
\begin{align*}
f(x+y)&=\sum_{s\in
\N_{\leq d}^n}(\mathcal{D}^{(s)}f)(x)y^s\\
&=f(x)+\sum_{1 \leq i \leq n}(\mathcal{D}_{x_i}f)(x)y_i+\cdots+\sum_{|s|=
d}(\mathcal{D}^{(s)}f)(x)y^s,
\end{align*}
with $y^s$ as in (\ref{powerProd}).

We write $f=\sum_{1 \leq i \leq d} f_i$, where each $f_i\in\fq[x_1,\ldots,x_n]$ is zero or
homogeneous of degree $i$, and similarly
$f(x+y)=\sum_{i=1}^df_i^*$, where each $f_i^*\in ( \fq[y_1,\ldots,y_n] )
[x_1,\ldots,x_n]$ is zero or homogeneous of degree $i$ in the $x_j$.
The Taylor formula for each $f_j(x+y)$ implies that
\begin{align*}
f_d^*&=f_d,\\ f_{d-1}^*&={f}_{d-1}+
\sum_{1 \leq i \leq n}(\mathcal{D}_{x_i}f_d)(x)y_i,\\&\ \,\vdots\\
f_0^*&={f}_0+
\sum_{1 \leq i \leq n} (\mathcal{D}_{x_i}f_1)(x)y_i+\cdots+\sum_{|s|=d}(\mathcal{D}^{(s)}f_d)(x)y^s.
\end{align*}
Now for $u=(u_1,\ldots,u_n) \in \fq^n$ and replacing each $y_i$ by $-u_i$ in the above,
we conclude that $f=f^{(u)}$ if and only if
\begin{equation}\label{eq:IdentityShift}
f_j={f}_j-
\sum_{1 \leq i \leq n} (\mathcal{D}_{x_i}f_{j+1})(x)u_i+\cdots+ (-1)^{d-j} \sum_{|s|=d-j}(\mathcal{D}^{(s)}f_d)(x)u^s
\end{equation}
for $0\le j\le d$.

\begin{lemma}
For $u=(u_1, \ldots, u_n) \in \fq^n$ and $j>0$, we have
\begin{equation}\label{eq:CompositionDiffOperatW}
\Big(\sum_{1 \leq i \leq n}u_i\mathcal{D}_{x_i}\Big)\circ\Big(\sum_{|t|=j-1}u^t\mathcal{D}^{(t)}\Big)=j
\sum_{|s|=j}u^s\mathcal{D}^{(s)}.
\end{equation}
\end{lemma}
\begin{proof}
Equation \eqref{eq:CompositionDiffOperators} shows that for $t \in \N^n$ we have
$$u_i\mathcal{D}_{x_i}\circ
u^t\mathcal{D}^{(t)}=u_iu^t(t_i+1)\mathcal{D}^{(\mathrm{e}_i+t)}
=(t_i+1)u^{\mathrm{e}_i+t}\mathcal{D}^{(\mathrm{e}_i+t)},$$
where $e_i \in \N^n$ is the $i$th unit vector.
We consider some $s =(s_1,\ldots,s_n) \in\N^n$ with $|s|=j$.
%Assume
%without loss of generality that $s$ with
%$s_1+\ldots+s_m=j$ and $s_{m+1}=\cdots=s_n=0$.
Then
$\mathcal{D}^{(s)}$ arises in the following sum on the left hand
side of \eqref{eq:CompositionDiffOperatW}:
$$\sum_{1 \leq i \leq n} u_i\mathcal{D}_{x_i}\circ u^{s-\mathrm{e}_i}\mathcal{D}^{(s-\mathrm{e}_i)}=
\sum_{1 \leq i \leq n} s_iu^{s}\mathcal{D}^{(s)}=ju^{s}\mathcal{D}^{(s)}.$$
This shows the claim.
\end{proof}

The following characterizes the shift-invariant
polynomials.
\begin{proposition}\label{coro:FirstCharactShiftInvariantPols}
Let $p>d$, $q$ a power of $p$, and
$f\in\fq[x_1,\ldots,x_n]\setminus\{0\}$ of degree $d$. Then $f$ is
shift-invariant under $u=(u_1,\ldots,u_n)\in\fq^n\setminus\{0\}$ if
and only if
\begin{equation}
\label{nabla}
\sum_{1 \leq i \leq n} u_i\mathcal{D}_{x_i}f=0.
\end{equation}
\end{proposition}
\begin{proof}
According to \eqref{eq:IdentityShift}, $f=f^{(u)}$ holds if
and only if
$$
- \sum_{1 \leq i \leq n} (\mathcal{D}_{x_i}f_{j+1})(x)u_i+ \cdots+  (-1)^{d-j} \sum_{|s|=d-j}(\mathcal{D}^{(s)}f_d)(x)u^s
= 0
$$
for $0\le j < d$. For $1\le k\le j$,
\eqref{eq:CompositionDiffOperatW} implies inductively that
$$
\sum_{|s|=j+k}u^s\mathcal{D}^{(s)}f_{j+k}=\frac{1}{(j+k)!}
\big(\sum_{1 \leq i \leq n} u_i\mathcal{D}_{x_i}\big)^{(j+k)}(f_{j+k}),
$$
where on the right hand side, the operator given by the sum is iterated $j+k$ times.
We deduce that, for $0\le j < d$,
\begin{equation}\label{eq:AuxShiftInvariancePols}
\big(-\sum_{1 \leq i \leq n} u_i\mathcal{D}_{x_i}\big)(f_{j+1})+
\cdots+\frac{(-1)^{d-j}}{(d-j)!}\big(\sum_{i=1}^nu_i\mathcal{D}_{x_i}\big)^{(d-j)}(f_d)
= 0.
\end{equation}

We claim that, for $1\le j\le d$,
$$\sum_{1 \leq i \leq n} u_i\mathcal{D}_{x_i}f_j=0.$$
Arguing by backward induction, the case $j=d$ is
\eqref{eq:AuxShiftInvariancePols} with $j=d-1$. Now suppose that the
assertion holds for any $k$ with $k>j\ge 1$. By
\eqref{eq:AuxShiftInvariancePols} we have
\begin{align*}
0&= \big(-\sum_{1 \leq i \leq n} u_i\mathcal{D}_{x_i}\big)(f_j)+
\cdots+\frac{(-1)^{d-j+1}}{(d-j+1)!}\big(\sum_{1 \leq i \leq n} u_i\mathcal{D}_{x_i}\big)^{(d-j+1)}(f_d)\\
&= \big(-\sum_{1 \leq i \leq n} u_i\mathcal{D}_{x_i}\big)(f_j),
\end{align*}
where the second identity is due to the inductive hypothesis. This
proves the claim.
Thus we have
$$\sum_{1 \leq i \leq n} u_i\mathcal{D}_{x_i}f=\sum_{1 \leq i \leq n} u_i\mathcal{D}_{x_i}\big(\sum_{1 \leq j \leq d} f_j\big)=
\sum_{1 \leq j \leq d} \big(\sum_{1 \leq i \leq n} u_i\mathcal{D}_{x_i}f_j\big)=0.$$
This finishes the proof of \eqref{nabla}.

On the other hand, if \eqref{nabla} holds, then by homogeneity it
follows that
$$\sum_{1 \leq i \leq n} u_i\mathcal{D}_{x_i}f_j=0$$
for $1\le j\le d$. This implies \eqref{eq:AuxShiftInvariancePols},
from which the $u$-shift-invariance of $f$ is readily deduced.
\end{proof}

This statement strengthens the corresponding one for
varieties (Corollary
\ref{coro:CharactShiftInvariantVarietiesWithDerivatives}) by providing a necessary and sufficient condition.
It also yields a polynomial-time algorithm for testing a polynomial for
(nontrivial) $\fq^n$-shift-invariance. Namely, (\ref{nabla})
corresponds to a system of linear equations in $u_1, \ldots, u_n$
with coefficients in $\fq[x_1,\ldots,x_n]$. Its size is polynomial
in the input size of $f$, given either in dense or sparse
representation. Its kernel
consists of all $u$
under which $f$ is shift-invariant,
and its triviality can be checked efficiently.

However, the problem of deciding $U_h$-shift-freeness turns out to be computationally hard,
namely coNP-complete under randomized reductions; see Theorem \ref{reductions} below.
Thus under standard complexity assumptions, no efficient algorithm for it exists.

We now provide an alternative statement and proof of Corollary
\ref{coro:CharactShiftInvariantVarietiesWithIdeals} in the special case where $X = \{ f=0 \}$
is a hypersurface.
\begin{proposition}\label{prop:SecondCharactShiftInvariantPols}
With hypotheses as in Proposition
\ref{coro:FirstCharactShiftInvariantPols}, %assume further that $f$
%is separable. Then
$f$ is shift-invariant if and only if there exist
linearly independent linear forms
$\ell_1,\ldots,\ell_{n-1}\in\fp[x_1,\ldots,x_n]$ and $g \in
\fp[y_1,\ldots,y_{n-1}]$ with new variables $y_1,\ldots,y_{n-1}$
such that $f=g(\ell_1,\ldots,\ell_{n-1})$.
\end{proposition}
\begin{proof}
First suppose that $f$ is invariant under a nonzero shift $u\in
\fp^n$. We assume without loss of generality that $u_n\not=0$ and
consider the linear invertible change of variables
$$x_1=y_1 + u_1y_n,\ x_2=y_2 + u_2y_n,\ldots,x_n=u_ny_n,$$
similar to \eqref{linForm}.
By the Chain rule we see that
%
%$$\left(\begin{array}{c} \frac{\partial f}{\partial
%y_1}\\\vdots\\\frac{\partial f}{\partial y_n}\end{array} \right)=
%\left(
%  \begin{array}{ccccc}
%    1 & u_2 & u_3 & \cdots & u_n \\
%       & 1 & 0  & \cdots & 0  \\
%     &   & & \ddots & \vdots \\
%     &  & &  & 1 \\
%  \end{array}
%\right)\left(\begin{array}{c} \frac{\partial f}{\partial
%x_1}\\\vdots\\\frac{\partial f}{\partial x_n}\end{array} \right).$$
%%
%
%The chain rule implies that
$\mathcal{D}_{y_i}f =\mathcal{D}_{x_i}f$ for $1 \leq i <n$, and
% In particular,
$$\mathcal{D}_{y_n}f=u_1\mathcal{D}_{x_1}f+u_2\mathcal{D}_{x_2}f+\cdots+
u_n\mathcal{D}_{x_n}f,
$$
which equals $0$ by Proposition
\ref{coro:FirstCharactShiftInvariantPols}. Since $\deg f=d<p$,
$f(y_1 + u_1y_n ,y_2 + u_2y_n,\ldots,u_n y_n)$ is separable in $y_n$,
that is, no exponent of $y_n$ is divisible by $p$. As a consequence,
$f(x_1,\ldots,x_n)=f(y_1 + u_1y_n,y_2 + u_2y_n,\ldots,u_n y_n)$ does not
depend on $y_n$, and is actually a polynomial $g \in
\fp[y_1,\ldots,y_{n-1}]$.
Thus $f = g(\ell_1,\ldots,\ell_{n-1})$ with $\ell_i = x_i - u_ix_n/u_n$ for $i<n$.

On the other hand, suppose that $f=g(\ell_1,\ldots,\ell_{n-1})$ with $g
\in \fp[y_1,\ldots,y_{n-1}]$ and linear forms $\ell_1,\ldots,\ell_{n-1}
\in\fp[x_1,\ldots,x_n]$. Then we can write $y=Ax$ with
$A\in\fp^{(n-1)\times n}$ and let $u\in\fp^n\setminus\{0\}$ be any
vector in the kernel of $A$. Then $f$ is invariant under the shift
$u$.
\end{proof}

We illustrate the application of Propositions
\ref{coro:FirstCharactShiftInvariantPols} and
\ref{prop:SecondCharactShiftInvariantPols} with examples of
hypersurfaces satisfying the hypotheses of Theorem
\ref{th:EstimateGenTubeVarieties} (\ref{shiftFreeBound}).

\begin{example}
Consider the graph $G=\{x_n=g(x_1,\ldots,x_{n-1})\}$ of $g \in
\fp[x_1,\ldots,x_{n-1}]$
with $p>d=\deg g \geq 2$. Let $f=x_n-g(x_1,\ldots,x_{n-1})$. Then
$\# G(\fp) = p^{n-1}$ and $f\in\fp[x_1,\ldots,x_n]$ is absolutely
irreducible with
$$\mathcal{D}_{x_i}f=-\mathcal{D}_{x_i}g \text{ for } 1\le i\le n-1 \text{ and } \mathcal{D}_{x_n} f=1.$$
We now show that $f$ is not shift-invariant. According to
Proposition \ref{coro:FirstCharactShiftInvariantPols}, we should
check if there exists $(u_1,\ldots,u_n)\in\fp^n\setminus\{0\}$ such
that
$$u_n=u_1\mathcal{D}_{x_1}g+\cdots+u_{n-1}\mathcal{D}_{x_{n-1}}g.$$
As $\deg g \geq 2$ and $g$ is separable, this condition is not
satisfied for any $(u_1,\ldots,u_n)\in\fp^n\setminus\{0\}$. Thus
Theorem \ref{th:EstimateGenTubeVarieties} (\ref{shiftFreeBound})
%Corollary \ref{coro:EstimateHypersurfacesAbsIrred}
implies
$$
\Delta \le \#U \#(U-U) \cdot d^2q^{r-1},
$$
and for $U=U_h$
\begin{align*}
|\# & (G(\fp) +U_h)-(2h+1)^n p^{n-1}|\\
& \le (2h+1)^n d^2 \big(p^{n-{3}/{2}}+
(5d^{{7}/{3}}+(4h+1)^n)p^{n-2}\big). \hfill \lozenge
\end{align*}
\end{example}

\begin{example}
\label{ex:notsquarefree}
For $p>n\geq 2$ and $q$ a power of $p$, consider the variety $\mathcal{S}_n$ of
univariate polynomials
$f = \sum_{0\leq i \leq n} a_ix^i \in\cfq[x]$
of degree at most $n$ that are not squarefree. Then $\mathcal{S}_n\subset\A^{n+1}$
is the hypersurface defined by the generic discriminant
$\mathrm{disc}_n\in\fq[a_0,\ldots,a_n]$. As $\mathrm{disc}_n$ is
absolutely irreducible (see, e.g., \cite[Th\'eor\`eme
1.7]{Benoist12}), we check that it is not shift-invariant.

If $f \in\fq[x]$
has a unique double root $\alpha \in \cfp$, then by \cite[Chapter 12,
Equation (1.28)]{GeKaZe94}\footnote{Although the identity is stated
in \cite{GeKaZe94} for the case of characteristic zero, it holds for
$p>2$.} (compare with \cite{Shparlinski15}), the following two
projective points are equal:
$$[1 :\alpha : \cdots : \alpha^n] =
\left[ (\mathcal{D}_{a_0} \mathrm{disc}_n) (f) :
(\mathcal{D}_{a_1}\mathrm{disc}_n) (f) : \cdots :
(\mathcal{D}_{a_n}\mathrm{disc}_n) (f)\right].$$

Now, if $u=(u_0,\ldots,u_n)\in\fp^{n+1}$ is such that
$$\nabla \mathrm{disc}_n\cdot u=0,$$
where $\nabla \mathrm{disc}_n$ is the gradient of
$\mathrm{disc}_n$, then in particular
\begin{equation}\label{eq:IdentityShiftInvarianceDiscriminant}
(1,\alpha ,\ldots, \alpha^n)\cdot u =0
\end{equation}
for any $\alpha\in\fp$, since there exists some polynomial with
$\alpha$ as its unique double root. As $p>n$, there exist pairwise
distinct elements $\alpha_0,\ldots,\alpha_n\in\fp$ for which
\eqref{eq:IdentityShiftInvarianceDiscriminant} is satisfied. We
conclude that $u$ is in the kernel of the $(n+1)\times(n+1)$
Vandermonde matrix defined by $\alpha_0,\ldots,\alpha_n$, which is
nonsingular. It follows that $u=0$, and Proposition
\ref{coro:FirstCharactShiftInvariantPols} implies that
$\mathrm{disc}_n$ is not shift-invariant.

%Therefore, taking into account that
Since $\deg \mathrm{disc}_n=2n-1 <2n$,
% and applying Corollary
%\ref{coro:EstimateHypersurfacesAbsIrred},
we obtain
$$
\Delta \le \#U \#(U-U) \cdot (2n)^2q^{n-1},
$$
and for $U=U_h \subseteq \fp$:
\begin{align*}
|\#(\mathcal{S}_n(\fp)&+U_h)-(2h+1)^{n+1} p^n|\\\le & 4n^2
(2h+1)^{n+1}\big(p^{n-{1}/{2}}+
(5(2n)^{{7}/{3}}+(4h+1)^{n+1})p^{n-1}\big). \hfill \lozenge
\end{align*}
\end{example}%

\begin{example}
\label{ex:approximate_gcd} Generalizing Example
\ref{ex:notsquarefree}, for $p>n+m+2$ and $q$ a power of $p$, we consider the variety
$\mathcal{S}_{n,m}$ of pairs of univariate polynomials $f =
\sum_{0\leq i \leq n} a_ix^i \in\cfq[x]$ and $g = \sum_{0\leq i \leq
m} b_ix^i \in\cfq[x]$ of degrees at most $n$ and $m$ that are not
coprime. It is well known that $\mathcal{S}_{n,m}\subset\A^{n+m+2}$
is the hypersurface defined by the generic resultant
$\mathrm{res}_{n,m}\in\fq[a_0,\ldots,a_n,b_0,\ldots,b_m]$. As
$\mathrm{res}_{n,m}$ is absolutely irreducible (see, e.g.,
\cite[Corollary 6.7.2]{Mora03}), we check that it is not
shift-invariant, using the approach of the previous example.

If $f,g \in\fq[x]$ have a unique common root $\alpha \in \cfq$, then by
\cite[Chapter 12, Equation (1.11)]{GeKaZe94}\footnote{Although the
identity is stated in \cite{GeKaZe94} for the case of characteristic
zero, it holds for $p>2$.}, we have the following identities of
projective points:
\begin{align*}
[1 :\alpha : \cdots : \alpha^n] &= \left[ (\mathcal{D}_{a_0}
\mathrm{res}_{n,m}) (f,g) :
\cdots : (\mathcal{D}_{a_n}\mathrm{res}_{n,m}) (f,g)\right],\\
[1 :\alpha : \cdots : \alpha^m] &= \left[ (\mathcal{D}_{b_0}
\mathrm{res}_{n,m}) (f,g) : \cdots :
(\mathcal{D}_{b_m}\mathrm{res}_{n,m}) (f,g)\right].
\end{align*}
In particular, considering suitable scalar multiples $\lambda f$ and
$\mu g$ of $f$ and $g$, with nonzero $\lambda,\mu\in\cfp$, we have
$$(1,\alpha,\ldots,\alpha^{n+m+1})=\nabla\mathrm{res}_{n,m} (\lambda f,\mu g)
= \lambda^m \mu^n \, \nabla\mathrm{res}_{n,m} (f, g),$$
where $\nabla \mathrm{res}_{n,m}$ is the gradient of
$\mathrm{res}_{n,m}$.

Now, if $u=(u_0,\ldots,u_{n+m+1})\in\fp^{n+m+2}$ is such that
$$\nabla \mathrm{res}_{n,m}\cdot u=0,$$
then in particular
\begin{equation}\label{eq:IdentityShiftInvarianceResultant}
(1,\alpha ,\ldots, \alpha^{n+m+1})\cdot u =0
\end{equation}
for any $\alpha\in\fp$, since there exist pairs of polynomials with
$\alpha$ as its unique common root. As $p>n+m+2$, there exist
pairwise distinct elements $\alpha_1,\ldots,\alpha_{n+m+2}\in\fp$
for which \eqref{eq:IdentityShiftInvarianceResultant} is satisfied.
We conclude that $u$ is in the kernel of the $(n+m+2)\times(n+m+2)$
Vandermonde matrix defined by $\alpha_1,\ldots,\alpha_{n+m+2}$,
which is nonsingular. It follows that $u=0$, and Proposition
\ref{coro:FirstCharactShiftInvariantPols} implies that
$\mathrm{res}_{n,m}$ is not shift-invariant.

%Therefore, taking into account that
Since $\deg \mathrm{res}_{n,m}=n+m$,
% and applying Corollary
%\ref{coro:EstimateHypersurfacesAbsIrred},
we obtain
$$
\Delta \le \#U \#(U-U) \cdot (n+m)^2q^{n+m},
$$
and for $U=U_h \subseteq \fp$:
\begin{align*}
|\#&(\mathcal{S}_{n,m}(\fp)+U_h)-(2h+1)^{n+m+2} p^{n+m+1}|\\\le &
(n+m)^2 (2h+1)^{n+m+2}p^{n+m}\big(p^{{1}/{2}}+
5(n+m)^{{7}/{3}}+(4h+1)^{n+m+2}\big). \hfill \lozenge
\end{align*}
\end{example}%

A natural question is the algorithmic aspect of shift-freeness:
can we determine efficiently whether a variety is $U$-shift-free?
Single polynomials can be tested for $\fq^n$-shift-freeness in polynomial time,
using \eqref{nabla}.
However, the neighborhood given by $U=\fq^n$ is not relevant in our context,
since  then $X+U = \fq^n$ for any $X$.

%Our case of interest includes the standard neighborhoods $U_h
%\subseteq \fp^n$ for primes $p$. Then
For more interesting neighborhoods $U$, the answer to the above
question is negative: the problem of determining
$U_h$-shift-freeness turns out to be coNP-complete under randomized
reductions.
%; see Theorem \ref{reductions} below.
This means that
under standard complexity assumptions, no efficient algorithm exists
for this task. This holds even for the special case where the
variety $X$ is a hyperplane $\{f=0\}$ with a linear $f$ and $U=U_h$ is a standard
neighborhood, already for $h=1$.

For a fixed $u\in \fq^n$ and hypersurfaces $X=\{f=0\}$ and
 $Y=\{g=0\}$ with squarefree polynomials $f$ and $g$ and leading coefficients
 (in some term order) $a$ and $b$, respectively,
we have $X = Y^{(u)}$ if and only if $b f(x) - a g(x-u) =0$.
This can be tested
in random polynomial time (with one-sided
error) by evaluating that difference at $x=c$ for uniformly random
points $c$ in a large enough finite subset of $\cfp^n$.
This works even in  a very concise presentation by a ``black box'' which produces
the values of polynomials in unit time.

Our starting point is the decision problem
{\sc Equal subset sum}, whose input is a sequence $(a_1, \ldots, a_n)$ of nonnegative integers
%each positive or negative and
presented in binary.
The task is to decide whether there exist two disjoint nonempty sets
$S, T \subseteq \{1,\ldots,n\}$ with $\sum_{i \in S} a_i = \sum_{i
\in T} a_i $. \cite{woeyu92} show that it is NP-complete. It is a
variant of {\sc Partition}, one of the ``original'' NP-complete
problems.

The variant {\sc Equal subset sum modulo prime} has $(a_1, \ldots, a_n)$ as above and a prime $p$ (in binary)
as input, and the question is again whether subsets $S$ and $T$ as above exists, now with
$\sum_{i \in S} a_i \equiv \sum_{i \in T} a_i \bmod p$.
%\todo{Li and Wang say subset sum modulo prime is NP-complete. I do not know a reference for that.
%If correct, it might imply that Equal subset sum mod p is also NP-complete.
%I suggest that we ignore this issue.}

Our interest is in the decision problem {\sc Non-shiftfreeness}.
We only consider the simple version with a standard neighborhood $U_h$ and a single polynomial
$f \in \fp[x_1, \ldots, x_n]$ of total degree at most $d$.
The input is presented by the prime $p$ and an integer $h$ with $p > 2h >2$ in binary,
$n$ and $d$ in unary,
and $f$ in dense representation. That is, for each exponent vector $(e_1, \ldots, e_n)$
with $\sum_{1 \leq i \leq n} e_i \leq d$,
the coefficient of $x^e$ in $f$ is given in binary.
The task is to decide whether there exists a nonzero $u\in U_h$ with $f = f^{(u)}$.

For two decision problems $A$ and $B$, we write $A \leq_p B$ if there exists a deterministic
polynomial-time reduction from $A$ to $B$, and $A \leq_r B$ if there is
some randomized polynomial-time reduction from $A$ to $B$.
The notion here is ``Las Vegas'', that is, the reduction returns either the correct answer
or ``fail''; the latter with probability at most $1/2$.
The corresponding complexity class is called ZPP (zero error probabilistic polynomial time).

\begin{theorem}
\label{reductions}
We have
\begin{align*}
\text{{\sc Equal subset sum}}& \leq_r \text{{\sc Equal subset sum modulo prime}} \\ & \leq_p \text{{\sc Non-shiftfreeness}}
\in \text{NP}.
\end{align*}
\end{theorem}
\begin{proof}
For the first reduction,
on input of nonnegative integers $(a_1, \ldots, a_n)$, we choose randomly a prime $p > \sum_{1\leq i \leq n} a_i$
and consider {\sc Equal subset sum modulo prime} with input $((a_1, \ldots, a_n),p)$.
The random choice is done by choosing integers larger than $b = \sum_{1\leq i \leq n} a_i$,
testing them deterministically for primality and accepting the first one that is certified to be prime.
Since the length of $b$ is polynomial in the input size, all this can be done error-free in polynomial expected time.
It is the only step where randomization intervenes.
If prime gaps were of polynomial size, this could even be done deterministically.

If $(S,T)$ is a solution for the {\sc Equal subset sum} instance, then it is also one for this instance of {\sc Equal subset sum modulo prime}.
On the other hand, suppose that $(S,T)$ is a solution for this {\sc Equal subset sum modulo prime} instance,
so that $\sum_{i \in S} a_i \equiv \sum_{i \in T} a_i \bmod p$.
We denote the two sums, taken as integers, as $b_S$ and $b_T$, respectively.
Then there exists an integer $k$ with $b_S - b_T = kp$ in $\Z$.
But $|b_S - b_T | \leq b < p$, so that $k=0$ and $(S,T)$ is also a solution for {\sc Equal subset sum}.

For $\text{{\sc Equal subset sum modulo prime}} \leq_p \text{{\sc Non-shiftfreeness}}$,
on input $((a_1, \ldots, a_n),p)$ with all $a_i \in \fp$, we take the linear form
$f = \sum_{1 \leq i \leq n} a_i x_i \in \fp[x_1,\ldots, x_n]$ and $U_1 = \{ 0,1,-1\}^n \subseteq \fp^n$.
Then $\mathcal{D}_{x_i}(f) = a_i$ for all $i$.
By Proposition \ref{coro:FirstCharactShiftInvariantPols}, $f$ is shift-invariant under some $u \in U_1$
if and only if $\sum_{1\leq i \leq n} a_i u_i = 0$.
We set up a bijection between $U_1$ and pairs of disjoint nonempty subsets
$S, T \subseteq \{1,\ldots,n\}$ by requiring for each $i \in \{1,\ldots,n\}$:
$$
i \in S \Longleftrightarrow u_i = 1; \quad i \in T \Longleftrightarrow u_i = -1.
$$
This bijection maps solutions $u$ of {\sc Non-shiftfreeness} to solutions $(S,T)$ of
{\sc Equal subset sum modulo prime}, and vice versa.
Furthermore, the reduction can be executed in deterministic polynomial time.

For $\text{{\sc Non-shiftfreeness}} \in \text{NP}$, we have some $u$ with $f = f^{(u)} = f(x-u)$.
Since the input $f$ is given in dense representation, we can compute the dense
representation of $f(x-u)$ in polynomial time, and then compare it to that of $f$.
\end{proof}

It follows that  {\sc Non-shiftfreeness} is NP-complete under randomized reductions,
and the natural complementary problem {\sc Shiftfreeness} is similarly coNP-complete.
Under standard complexity assumptions, no efficient algorithm for it exists.

\begin{center}
{\sc Open questions}
\end{center}

When $u\in \fq$ and two varieties $X$ and $Y$ are given, can we test efficiently whether
$X=Y^{(u)}$? When $X=Y$?
For hypersurfaces, this is feasible; see above.

When $Y$ is absolutely irreducible, $u \in U$ nonzero, and $X = Y \cup Y^{(u)}$,
what can we say about question (\ref{question})?

\vspace{1ex}

\begin{center}
{\sc Acknowledgements}
\end{center}

Many thanks go to Igor Shparlinski who mentioned this problem to us, on a visit by JvzG to Sydney.
We are grateful to Heiko R\"oglin for pointing the reference \cite{woeyu92}
out to us.

%\begin{thebibliography}{10}
%\expandafter\ifx\csname url\endcsname\relax
%  \def\url#1{\texttt{#1}}\fi
%\expandafter\ifx\csname urlprefix\endcsname\relax\def\urlprefix{URL
%}\fi \expandafter\ifx\csname href\endcsname\relax
%  \def\href#1#2{#2} \def\path#1{#1}\fi
  %\bibliographystyle{apalike}
% \bibliographystyle{elsarticle-harv}
% \bibliographystyle{elsarticle-num-names}
%\bibliography{biblio,journals}

\end{document}